\documentclass[sn-mathphys,Numbered]{sn-jnl}

\usepackage{xcolor}%
\usepackage{natbib}
\usepackage{color}
\usepackage{graphicx}%
\usepackage{multirow}%
\usepackage{amsmath,amssymb,amsfonts}%
\usepackage{amsthm}%
\usepackage{mathrsfs}%
\usepackage[title]{appendix}%
\usepackage{textcomp}%
\usepackage{manyfoot}%
\usepackage{booktabs}%
\usepackage{algorithm}%
\usepackage{algorithmicx}%
\usepackage{algpseudocode}%
\usepackage{listings}%

\theoremstyle{thmstyleone}%
\newtheorem{theorem}{Theorem}
\theoremstyle{thmstyletwo}%
\newtheorem{remark}{Remark}%
\newtheorem{lemma}{Lemma}
\theoremstyle{thmstylethree}%
\newtheorem{definition}{Definition}%

\begin{document}

\title[On the steadiness of symmetric solutions to two dimensional dispersive models]{On the steadiness of symmetric solutions to two dimensional dispersive models}


\author*[1]{\fnm{Long} \sur{Pei}}\email{peilong@mail.sysu.edu.cn}

\author[1]{\fnm{Fengyang} \sur{Xiao}}\email{xiaofy5@mail2.sysu.edu.cn}

\author[1]{\fnm{Pan} \sur{Zhang}}\email{zhangp273@mail2.sysu.edu.cn}


\affil*[1]{\orgdiv{School of Mathematics (Zhuhai)}, \orgname{Sun Yat-sen University}, \orgaddress{\city{Zhuhai}, \postcode{519082}, \state{Guangdong}, \country{China}}}

\abstract{In this paper, we consider the steadiness of symmetric solutions to two dispersive models in shallow water and hyperelastic mechanics, respectively. These models are derived previously in the two-dimensional setting and can be viewed as the generalization of the Camassa-Holm and Kadomtsev-Petviashvili equations. For these two models, we prove that symmetry of classical solutions  implies  steadiness in the horizontal direction.  We also confirm the such connection between symmetry and steadiness in weak formulation, which includes in particular the peaked solutions.}





\maketitle

\section{Introduction}\label{sec1}

Of main concern in this paper is the  Camassa-Holm-Kadomtsev-Petviashvili (CH-KP) equation, proposed in \cite{Gui2021} for  two-dimensional shallow-water waves propagating over a flat bed.  This equation has the form
\begin{equation}\label{eq:chkp}
	u_{xt}-\frac{5}{12} \gamma u_{xxxt} + u_{xx}+\frac{3}{2} \varepsilon (u u_x)_x -\frac{1}{4} \gamma u_{xxxx} -\frac{5}{24} \gamma \varepsilon (2u_x u_{xx}+u u_{xxx})_x+\frac{1}{2} \varepsilon^3 u_{yy}=0,
\end{equation}
where  $\varepsilon$ and $\gamma$ are small parameters that measure the nonlinearity and dispersion, respectively. It characterizes the weak transversal effects in two-dimensional flows, as the Kadomtsev-Petviashvili  equation \cite{kadomtsev1970stability} does, but exhibits stronger nonlinearity. It is worth to mention that  \eqref{eq:chkp} is closely related to the  following nonlinear dispersive equation

\begin{equation}\label{eq:hcp}
	(u_{t}-u_{xxt}+3uu_{x}-\gamma(2u_{x}u_{xx}+uu_{xxx}))_{x}-\alpha u_{yy}+\beta u_{xxyy}=0,
\end{equation}
where $\alpha,\beta,\gamma\in[0,\infty)$ are material parameters. 
However, the equation \eqref{eq:hcp} is originally derived by Chen  \cite{MR2273493}  to model the deformation  of a hyperelastic compressible plate relative to a uniformly pre-stressed state. 

\medskip

The model \eqref{eq:chkp} is derived within the Camassa-Holm (CH)  regime, and can  be viewed as a two-dimensional Camassa-Holm equation \cite{Fuchssteiner1981,MR1234453}. It is then expected that the solution to \eqref{eq:chkp} preserves some properties as that to the latter. For Cauchy problem of  \eqref{eq:chkp}, local well-posedness is confirmed in  Sobolev-type space $X^{s}$, $s\geq 2$, which is clearly introduced in Remark \ref{X^s}, while the full gradient  and the $x$-derivative of the solution blow up within finite time \cite{Gui2021} (see also \cite{Ming2023}). Moreover, \eqref{eq:chkp} allows for smooth and peaked traveling solutions \cite{Gui2021,Moon2022}. This is in line with the results for the Camassa-Holm equation, which allows for both smooth \cite{LENELLS2005393} and peaked solitary waves \cite{Alber1994,Camassa1993}, the latter being confirmed to be orbitally stable in $H^{1}(\mathbb{R})$ \cite{Constantin2000}. It is worth to mention that the Camassa-Holm equation is  well-posed in the critical Besov space $B^{\frac{2}{3}}_{2,1}$, but presents the wave-breaking phenomenon for smooth, sufficiently asymmetric initial data (see  \cite{Danchin2003,Constantin2000b} and references therein.)

\medskip

In this paper, we consider the steadiness of symmetric solutions to \eqref{eq:chkp} and its formal generalization \eqref{eq:hcp}. The symmetry of traveling solution has been a topic of great concern in water wave problems long time ago. It is known that symmetry holds \emph{a priori} for  irrotational flows \cite{Garabedian1965} and flows with vorticity \cite{Constantin2004a} if the global or local monotone structure is imposed (see also \cite{Okamoto2001,MikyoungHur2007} and references therein). Recently,  periodic traveling solutions to a large family of weakly dispersive equations are also confirmed in \cite{Bruell2023} to be  symmetric when reflection criterion holds. This criterion does not impose monotone structure at any point in a period. A precise characterization of the connection between symmetry and steadiness is then an intriguing topic. The first study in this direction is due to \cite{Ehrnstroem2009}, where a general principle for symmetric solutions to be steady was put forward for a large family of evolution equations. This work is further generalized in \cite{Bruell2017}, where the steadiness of symmetric solutions is thoroughly studied for general evolution equations in higher dimensions and differential systems via three principles. Recently, the principle in  \cite{Ehrnstroem2009} was applied to confirm the steadiness of symmetric solutions to unidirectional shallow water models, like the Whitham equation \cite{Bruell2017},  Degasperis-Procesi equation \cite{Pei2023}, Rotation-Camassa-Holm equation \cite{Wang2021} and some highly nonlinear one-dimensional shallow water models with moderate  and large amplitude \cite{Geyer2015,MR4564051}. In particular,  an alternative argument is put forward in \cite{Pei2023} to uncover more precise information about the connection between symmetry and steadiness of solutions. 

\medskip

{Note that the models of concern in this paper are two-dimensional, but we focus on solutions which are symmetric with respect to the $x$-variable. For smooth  solutions with symmetry in the $x$-direction, we adjust the argument in \cite{Bruell2017} so that it applies in the two dimension setting. For solutions in weak formulation,  we generalize the argument in \cite{Ehrnstroem2009, Geyer2015} from one dimension to two dimensions, and confirm that weak solutions with symmetry in the $x$-direction are also steady solutions in the $x$-direction. These weak solutions with symmetry in the $x$-direction correspond to but are more general than the peaked solitary waves considered in \cite{Gui2021}. 

\medskip 

We conclude this section with the frame of this paper. In section \ref{sect:preliminary}, we first introduce functions with  symmetry and steadiness in the $x$-direction. Then, we prove that smooth, $x$-symmetric solutions are actually traveling waves in the $x$-direction for the CH-KP equation in Sect. \ref{sect:chkp} and for the hyperelastic compressible plate model  \eqref{eq:hcp} in Sect. \ref{sect:hcp}. In Sect. 5, we introduce the weak formulation of solutions, and prove that $x$-symmetric weak solutions are  steady in the $x$-direction. 

\medskip

\section{Preliminary}\label{sect:preliminary}
Let $I$ be an interval of existence of a given equation, which is usually taken as $[0,T]$ with $T>0$. We first introduce the concept of a function $f(t,x,y)$ with $(t,x,y)\in I\times\mathbb{R}^{2}$ to be symmetric with respect to the variable $x$. 

\medskip

\begin{definition}\label{def:x-symmetry}
	We say a function $f(t,x,y)$, $(t,x,y)\in I\times\mathbb{R}^{2}$, is  $x$-symmetric  if there exists a function $\lambda(t) \in C^{1}(\mathbb{R})$ such that
	\begin{equation}
		f(t,x,y)=f(t,2\lambda(t)-x,y) \nonumber
	\end{equation}
	for  $(t,x,y)\in I\times\mathbb{R}^{2}$. We call $\lambda=\lambda(t)$  the \emph{axis of symmetry}.
\end{definition}
We also need the following concept of a function $f(t,x,y)$ with $(t,x,y)\in I\times\mathbb{R}^{2}$ to be steady in the $x$-direction. 

\medskip 

\begin{definition}
	We say a function $f(t,x,y)$, $(t,x,y)\in I\times\mathbb{R}^{2}$, is  steady in the  $x$-direction if there exists a function $g$ such that    
	\begin{equation}
		f(t,x,y)=g(x-ct,y) \nonumber
	\end{equation}
	for  $(t,x,y) \in I\times\mathbb{R}^{2}$.
\end{definition}

We first reformulate \eqref{eq:chkp} in a simpler form as in \cite{Gui2021}.  In view of the scale transformation 
\begin{align}
	t &\rightarrow \sqrt{\frac{5}{12}} \gamma t , \quad x \rightarrow  \sqrt{\frac{5}{12}} \gamma x,\quad y \rightarrow \varepsilon^{-\frac{3}{2}}\sqrt{\frac{5}{6} \gamma}y,\notag \\
	u(t,x,y) &\rightarrow \frac{1}{2} \varepsilon u(t,x-(\frac{3}{4}+\kappa)t,y)+\frac{1}{4},\notag
\end{align}
we can 
rewrite \eqref{eq:chkp} in the following form
\begin{equation}\label{eq:chkp reformulation}
	Lu_t+\kappa u_{xx}+(3uu_x-2u_xu_{xx}-uu_{xxx})_x+u_{yy}=0,
\end{equation}
where $\kappa \in \mathbb{R}$ and the operator $L$ is defined by $Lf:=\partial_x(1-\partial_x^2)f$.

\section{Steadiness of $x$-symmetric solutions to  Camassa-Holm-Kadomtsev-Petviashvili equation}\label{sect:chkp}

In this section, we prove that  classical solutions to the Camassa-Holm-Kadomtsev-Petviashvili equation \eqref{eq:chkp reformulation} with symmetry in $x$-direction are steady in the $x$-direction. The precise formulation of the result is as follows.

\medskip 

\begin{theorem}\label{thm:chkp}
	Assume that the Camassa-Holm-Kadomtsev-Petviashvili equation \eqref{eq:chkp reformulation} has a unique classical solution $u(t,x,y)$, $(t,x,y)\in I\times\mathbb{R}^{2}$, for given initial data $u_{0}(x,y)=u(0,x,y)$. If $u(t,x,y)$ is $x$-symmetric, then it is a steady solution in the $x$-direction with speed $\dot{\lambda}(0)$. 
\end{theorem}
\begin{proof}
	Assume that the CH-KP equation \eqref{eq:chkp reformulation}
	has a symmetric solution with respect to the component $x$, which satisfies $u(t,x,y)=u(t,2\lambda(t)-x,y)$ for some function $\lambda(t) \in C^1(\mathbb{R})$. Then we have
	\begin{align}
		u_{t}(t,x,y)&=(u_{t}+2 \dot{\lambda} u_{x})(t,2\lambda-x,y),\label{pt}\\
		\partial_{x}^{n}u(t,x,y)&=(-1)^{n}\partial_{x}^{n}u(t,2\lambda-x,y),\quad n \in \mathbb{Z}^+. \label{px}
	\end{align}
	Inserting \eqref{pt} and \eqref{px} into \eqref{eq:chkp reformulation}
	and in view of the arbitrariness of  $(t,x,y)\in I\times\mathbb{R}^{2}$, we get
	\begin{equation}\label{eq:Transformed chkp}
		-L(u_t+2\dot{\lambda} u_x) +\kappa u_{xx}-(-3uu_x+2u_xu_{xx}+uu_{xxx})_x+u_{yy}=0.
	\end{equation}
	Comparing \eqref{eq:chkp reformulation} with \eqref{eq:Transformed chkp}, we obtain the following equations
	\begin{align}
		L(u_t+\dot{\lambda} u_x)&=0,\label{eq:shape of wave}\\
		-L(\dot{\lambda} u_x) +\kappa u_{xx}+(3uu_x-2u_xu_{xx}-uu_{xxx})_x+u_{yy}&=0.\label{eq:velocity of wave}
	\end{align}
	It is clear that \eqref{eq:shape of wave} is a linear equation and allows for solutions of the form 
	\begin{equation}\label{solution with fixed shape}
		u(t,x,y)=g(x-\lambda(t),y)
	\end{equation}
	for some function $g$. We now choose arbitrarily  two points  $(t_1,x_1,y),(t_2,x_2,y) \in I \times \mathbb{R}^{2}$ 
	such that
	$$
	x_1-\lambda(t_1)=x_2-\lambda(t_2)=:X .
	$$
	Inserting \eqref{solution with fixed shape} into \eqref{eq:velocity of wave}, and evaluating the latter at $(t_1,x_1,y)$ and $(t_2,x_2,y)$, respectively, we get
	\begin{align}
		& L(\dot{\lambda}(t_1)g_{x}(X,y))-M(X,y)=0,\label{eq:pair xt1}\\
		& L(\dot{\lambda}(t_2)g_{x}(X,y))-M(X,y)=0,\label{eq:pair xt2}
	\end{align}
	where $M(X,y):=(\kappa g_{xx}+(3g g_{x}-2g_{x}g_{xx}-gg_{xxx})_x+g_{yy})(X,y)$.
	Subtracting \eqref{eq:pair xt1} by \eqref{eq:pair xt2}  gives
	$$
	(\dot{\lambda}(t_1)-\dot{\lambda}(t_2))Lg_{x}(X,y)=0.
	$$
	Since $(X,y)$ can be taken arbitrarily, we have $Lg_{x}(X,y) \not\equiv 0$ so that $\dot{\lambda}(t)$ is a constant for any $t \in I$. Therefore, the solution in \eqref{solution with fixed shape} takes the form $u(t,x,y)=g(x-ct,y)$ for some constant $c\in \mathbb{R}$. Now, we take $g$ to be $u_{0}$ and define  $v(t,x,y)=u_{0}(x-ct,y)$. It is clear that $v(t,x,y)$ satisfies \eqref{eq:shape of wave} and \eqref{eq:velocity of wave}. Hence, it satisfies the CH-KP equation \eqref{eq:chkp reformulation}, and therefore is a solution. In addition, we have $v(0,x,y)=u_{0}(x,y)$ so that the initial condition is satisfied. The uniqueness of the solution then implies that $u(t,x,y)= v(t,x,y)=u_{0}(x-ct,y)$, which implies that the $x$-symmetric solution $u(t,x,y)$ is a traveling wave with respect to the $x$-variable.
	Therefore, any classical $x$-symmetric solution of the CH-KP equation is steady in the $x$-direction.
\end{proof}

\begin{remark}\label{X^s}
	Note that local well-posedness of \eqref{eq:chkp reformulation} has been confirmed in \cite{Gui2021}  in Sobolev type spaces $X^{s}$, $s\geq2$, which is defined by 
	\begin{equation}
		X^{s}=X^{s}(\mathbb{R}^{2}):=\{u\in H^{s}(\mathbb{R}^{2})| \partial_{x}u\in H^{s}(\mathbb{R}^{2}),  \partial^{-1}_{x}u\in H^{s}(\mathbb{R}^{2})\}. \nonumber
	\end{equation}
	Therefore, the uniqueness assumption of classical solutions in Theorem \ref{thm:chkp} is guaranteed when the $x$-symmetric classical solution is taken from the space $C(I;X^{s}(\mathbb{R}^{2}) \cap C^{1}(I;X^{s-2}(\mathbb{R}^{2}))$, $s\geq2$.
\end{remark}

\section{Steadiness of $x$-symmetric solutions to the hyperelastic compressible plate model}\label{sect:hcp}

We now consider the steadiness of $x$-symmetric solutions to the hyperelastic compressible plate model \eqref{eq:hcp}. It is intriguing to know that the higher order dispersive term $u_{xxyy}$ may change the integrable structure of the equation, therefore affecting the existence of solitons as special steady solutions, but all $x$-symmetric solutions are still steady in the $x$-direction. This is precisely stated in the following theorem. 

\medskip 

\begin{theorem}\label{thm:hcp}
	Assume that \eqref{eq:hcp} has a unique classical solution $u(t,x,y)$, $(t,x,y)\in I\times\mathbb{R}^{2}$, for given initial data $u_{0}(x,y)=u(0,x,y)$. If $u(t,x,y)$ is $x$-symmetric, then it is a steady solution in the $x$-direction with speed $\dot{\lambda}(0)$.
\end{theorem}

\begin{proof}
	We will sketch the proof since the idea is similar to that in the proof for Theorem \ref{thm:chkp}. Assume that  \eqref{eq:hcp} has an $x$-symmetric solution, that is $u(t,x,y)=u(t,2\lambda-x,y)$ for some differentiable function $\lambda(t)$, $t\in I$. Then, we observe that $u$ satisfies  
	\begin{align}
		L(u_t+\dot{\lambda} u_x)&=0,\label{eq:shape of wave for hcp}\\
		-\dot{\lambda}L u_x + (3uu_x-\gamma(2u_xu_{xx}+uu_{xxx}))_x-\alpha u_{yy} +\beta u_{xxyy}&=0.\label{eq:velocity of wave for hcp}
	\end{align}
	In particular, \eqref{eq:shape of wave for hcp} indicates the existence of solutions with the special form $u(t,x,y)=h(x-\lambda(t),y)$ for some function $h$.
	We now choose arbitrarily  two points  $(t_1,x_1,y),(t_2,x_2,y) \in I \times \mathbb{R}^{2}$ 
	such that
	$
	x_1-\lambda(t_1)=x_2-\lambda(t_2)=:X.
	$
	Inserting the solution $h(x-\lambda(t),y)$ into  \eqref{eq:velocity of wave for hcp}, and evaluating it at   $(t_1,x_1,y)$ and $(t_2,x_2,y)$, we get
	$$
	(\dot{\lambda}(t_1)-\dot{\lambda}(t_2))Lh_{x}(X,y)=0.
	$$
	In view of the arbitrariness of $(X,y)\in \mathbb{R}^{2}$, we conclude that  $\dot{\lambda}(t)=c$, $t \in I$, for some constant $c\in\mathbb{R}$. Now, we take the function $h$ to be $u_{0}$ and define  $v(t,x,y)=u_{0}(x-ct,y)$. Then  $v(t,x,y)$ satisfies the evolution equation \eqref{eq:hcp} and the initial condition is satisfied. The uniqueness of the solution implies that $u(t,x,y)=v(t,x,y)=u_{0}(x-ct,y)$, and therefore $u(t,x,y)$ is steady in the $x$-direction.
	Hence, any classical $x$-symmetric solution to the hyperelastic compressible plate model \eqref{eq:hcp} is steady in the $x$-direction.
\end{proof}

\begin{remark}
	The existence and uniqueness of solutions to  \eqref{eq:hcp} have been confirmed in \cite{MR2463973}  in Bourgain type space $X^{b,0}$, for some $b>\frac{1}{2}$, which is defined by 
	\begin{equation}
		\|u\|_{X^{b,0}}=\|\left\langle\tau-p(\xi,\eta)\right\rangle^{b}\left\langle\xi\right\rangle^{2}\hat{u}(\tau,\xi,\eta)\|_{L^{2}_{\tau\xi\eta}} , \nonumber
	\end{equation}
	where 
	\begin{equation}
		p(\xi,\eta)=\frac{\alpha\xi^{-1}\eta^{2}+\beta\xi\eta^{2}}{1+\xi^{2}}, \quad \alpha,\beta\in[0,\infty).\nonumber
	\end{equation}
	Therefore, the uniqueness assumption of classical solutions in Theorem \ref{thm:hcp} is guaranteed when the $x$-symmetric classical solution is taken from the space $X^{b,s}$ with $s\geq 0$ and $b>\frac{1}{2}$.
\end{remark}

\section{Steadiness of $x$-symmetric solutions in weak formulation}
The existence of peaked solitary waves, like peakons, is an important feature of the interaction between  relatively weak dispersion and strong nonlinearity. Moreover, the existence of smooth solitary waves is excluded for some dispersive models, like the KP-II equation \cite{MR2273493}. Therefore, exploring the existence and properties of traveling wave solutions in weak formulation are both of significance for the CH-KP equation and with challenge due to the existence of slow transverse effects that can cause the direction of the wave solutions to be deflected or the shape of the wave to change. A criterion for the existence of peaked solitary wave solutions to the CH-KP equation is confirmed in \cite{Gui2021}. These peaked solitary wave solutions are weak, steady solutions  in both the $x$ and $y$ directions  of the form 
\begin{equation}\label{eq: specific peak solitary solution of CH-KP}
	u(t,x,y)=a e^{-|x+\beta y -ct|}, \quad c\in \mathbb{R}.
\end{equation}
In this section, we consider the steadiness  of weak solutions with symmetry in the horizontal direction for both the CH-KP equation and the hyperelastic compressible plate model. We first introduce the weak formulation of solutions to the CH-KP equation. 

\medskip 

\begin{definition}\label{def:weak}
A function $u \in C(I,H^{1}(\mathbb{R}^{2}))$ is called a weak solution of \eqref{eq:chkp reformulation} if $u $ satisfies
\begin{equation}\label{def:chkp weak solution}
	\iiint_{I \times \mathbb{R}^2} u L(\partial_t\phi) +(\kappa u + \frac{3}{2} u^{2} +\frac{1}{2} (\partial_x u)^{2}) \partial_x^2\phi - \frac{1}{2} u^{2} \partial_x^4 \phi+u \partial_y^2\phi \, \mathrm{d}t\mathrm{d}x\mathrm{d}y=0
\end{equation}
for all $\phi \in C^{\infty}_{0}(I \times \mathbb{R}^{2})$.
\end{definition}
In order to characterize the steadiness of the weak solutions, we give the following lemma.
\begin{lemma}\label{chkp-lemma}
    If $U(x,y)\in  H^{1}(\mathbb{R}^{2}) $ satisfies
    \begin{equation}\label{equation lemma}
        \iint_{\mathbb{R} \times \mathbb{R}} -c U L(  \partial_x \psi) +(\kappa U + \frac{3}{2} U^{2} +\frac{1}{2} (\partial _{x}U)^{2})\partial^2_{x} \psi -\frac{1}{2} U^{2} \partial_x^4 \psi+U \partial_y^2 \psi \,\mathrm{d}x\mathrm{d}y=0,
    \end{equation}
    for all $\psi \in C^{\infty}_{0}(\mathbb{R}^2) $, then the function $u$ given by 
    \begin{equation} \label{eq:weak solution}
        u(t,x,y)=U(x-c(t-t_0),y)
    \end{equation}
    is a weak solution of \eqref{eq:chkp reformulation} for any $t_0 \in I$.
\end{lemma}

\begin{remark}
	It is worth to mention that the peaked solitary  wave solution \eqref{eq: specific peak solitary solution of CH-KP} and the periodic peaked traveling solution found in \cite{Moon2022} are both covered by the form \eqref{eq:weak solution}.
\end{remark}

\begin{proof}[Proof of Lemma \ref{chkp-lemma}]
    By using the Fourier transform, it is clear that the translation map $a \mapsto U(x+a,y)$ is continuous $\mathbb{R} \rightarrow H^{1}(\mathbb{R}^{2})$ , Since $t \mapsto c(t-t_{0})$ is real analytic, it thus follows that $u$ given by \eqref{eq:weak solution} belongs to $C(I ,H^{1}(\mathbb{R}^{2}))$.
        Recall that the weak solution $u$ satisfies
  \begin{equation*}
		\langle u,L(\varphi_{t}) \rangle+\langle \kappa u +\frac{3}{2}  u^{2} +\frac{1}{2}  (\partial_x u)^{2} ,\partial_x^2 \varphi \rangle -  \langle \frac{1}{2} u^{2} , \partial_x^4\varphi \rangle+ \langle u, \partial_y^2\varphi \rangle=0.
	\end{equation*}
	For any function $\varphi \in C^{\infty}_{0}(I \times \mathbb{R}^2)$, we have that
    \begin{equation}\label{eq:variable substitution c}
        \langle u,\varphi \rangle=\langle U,\varphi_c \rangle, \quad \langle u^2,\varphi \rangle =\langle U^2,\varphi_c \rangle, 
         \quad \langle u_x^{2} ,\varphi \rangle =\langle U_x^{2},\varphi_c \rangle,
    \end{equation}
    where we denote 
$         
\varphi_c(t,x,y): = \varphi(t,x+c(t-t_0),y) .
$
 By direct calculation, we get
    \begin{equation}\label{eq:derivative wrt phi c}
        (\varphi_c)_t = (\varphi_t)_c+ c(\varphi_x)_c, \quad        (\varphi_c)_x = (\varphi_x)_c,\quad  (\varphi_c)_y = (\varphi_y)_c.
    \end{equation}
With  \eqref{eq:variable substitution c} and \eqref{eq:derivative wrt phi c}, we obtain
    \begin{equation*}\label{eq:variable substitution c weak}
        \langle u,L(\varphi_{t})\rangle 
              = \langle U, L(\partial_t \varphi_c -c \partial_x \varphi_c)\rangle,
              \quad  \langle u^2,\partial_x^4\varphi\rangle  = \langle U^2,\partial_x^4\varphi_c \rangle,
                \quad \langle u,\partial_y^2\varphi\rangle= \langle U , \partial_y^2\varphi_c\rangle
    \end{equation*}
    and
    \begin{equation*}\label{eq:variable substitution c weak 1}
        \langle\kappa u +\frac{3}{2}u^2+\frac{1}{2} (\partial_x u)^2,\partial_x^2\varphi\rangle =\langle \kappa U+\frac{3}{2}U^2 +\frac{1}{2}(\partial_x U)^2,\partial_x^2 \varphi_c \rangle.
    \end{equation*}
    Since $U$ is independent of time, we deduce that
        \begin{equation*}
      \langle U, L(\partial_t \varphi_c )\rangle 
        =\iint_{ \mathbb{R}^{2}} U(x,y)L(\varphi_c(T,x,y)-\varphi_c(0,x,y) \, \mathrm{d} x \mathrm{d}y =0
    \end{equation*}
 Collecting the above results, we conclude that
\begin{align*}
    &\langle u,L(\varphi_{t}) \rangle+\langle \kappa u +\frac{3}{2}  u^{2}+\frac{1}{2}  (\partial_x u)^{2}, \partial_x^2\varphi \rangle -  \langle \frac{1}{2} u^{2} , \partial_x^4\varphi \rangle+ \langle u, \partial_y^2\varphi \rangle \nonumber \\
    =&\iiint_{I \times \mathbb{R}^{2} } [-c U(x,y)L(\partial_{x}\varphi_c)+(\kappa U(x,y)+\frac{3}{2}U(x,y)^2\nonumber\\
    &+\frac{1}{2} (\partial_x U(x,y))^2)\partial_x^2\varphi_c-\frac{1}{2}U(x,y)^2\partial_x^4\varphi_c+U(x,y)\partial_y^2\varphi_c] \, \mathrm{d}t\mathrm{d}x\mathrm{d}y \\
    =&0,
\end{align*}
where in the last equality we used  \eqref{equation lemma} by taking $\psi=\varphi$ which is in $ C^{\infty}_{0}$, for each given $t$. The lemma then follows directly.
\end{proof}

\begin{theorem}\label{the:weak chkp}
	Let $u$ be a weak solution of the Camassa-Holm-Kadomtsev-Petviashvili equation \eqref{eq:chkp reformulation} with initial data $u_{0}(x,y)=u(t_{0},x,y)$ such that the equation is locally well-posed \cite{chen3032}. If $u$ is x-symmetric, then it is a traveling wave in the x-direction.
\end{theorem}

\begin{proof}
Consider test functions $\phi \in C^{1}_{0}(I,C^{\infty}_{0}(\mathbb{R}^{2}))$.

Firstly, for $\lambda(t) \in C^1(I)$, we will  prove the transform $T_{\lambda}:\phi \rightarrow \phi_{\lambda}=\phi(t,2\lambda(t)-x,y)$ is a bijection in $C^{1}_{0}(I,C^{\infty}_{0}(\mathbb{R}^{2}))$.
Since we have $T_{\lambda}(\phi_{\lambda})=\phi(t,2\lambda(t)-(2\lambda(t)-x),y)=\phi(t,x,y)=\phi$, which indicates that $T_{\lambda}$ maps the space $C^{1}_{0}(I,C^{\infty}_{0}(\mathbb{R}^{2}))$ onto the same space for $ \lambda \in C^{1}(I)$.
Moreover, for any $\phi \in C^{1}_{0}(I,C^{\infty}_{0}(\mathbb{R}^{2}))$, there exists $\Tilde{\phi}=\phi_{\lambda}=\phi(t,2\lambda-x,y)$ such that $T_{\lambda} \Tilde{\phi}=\phi$ This demonstrates the surjectivity of $T_{\lambda}$.
Additionally, if $T_{\lambda}(\phi)=T_{\lambda}(\psi)$ for $\phi,\psi \in C^{1}_{0}(I,C^{\infty}_{0}(\mathbb{R}^{2}))$, $\mathrm{ i.e. } \, \phi_{\lambda}=\psi_{\lambda}$, then $T_{\lambda}(\phi_{\lambda})=T_{\lambda}(\psi_{\lambda}) \Rightarrow \phi=\psi$. This confirms the injectivity of $T_{\lambda}$, concluding it as a bijection.

Let $u$ be an $x$-symmetric solution of  \eqref{eq:chkp reformulation} in the sense of Definition \ref{def:x-symmetry}, namely $u=u_{\lambda}$. It is clear that	
\begin{equation}\label{eq:derivative wrt symmetry}
\langle u_{\lambda},\phi \rangle= \langle u, \phi_{\lambda} \rangle, \quad \langle \partial_x(u_{\lambda}),\phi \rangle= -\langle \partial_x u, \phi_{\lambda} \rangle.
\end{equation}
Using \eqref{eq:derivative wrt symmetry} and inserting $u$ into the weak solution formulation \eqref{def:chkp weak solution}, we get
	\begin{equation}\label{eq:symmetry weak solution phi-lambda}
		0=\langle u , (L(\phi_t))_{\lambda} \rangle + \langle \kappa u +\frac{3}{2}  u^{2} +\frac{1}{2}   (\partial_x u)^{2} , (\partial_{x}^{2}\phi)_{\lambda} \rangle -  \langle \frac{1}{2} u^{2} , (\partial_{x}^{4}\phi)_{\lambda} \rangle+ \langle u, \partial_{y}^{2}(\phi_{\lambda}) \rangle  .
	\end{equation}
	Note that 
	\begin{equation}
			(\phi_{\lambda})_{t}=(\phi_{t})_{\lambda}-2 \dot{\lambda} (\phi_{x})_{\lambda},\quad (\phi_{\lambda})_{x}=-(\phi_{x})_{\lambda}, \quad (\phi_{\lambda})_{y}=(\phi_{y})_{\lambda}  \nonumber
	\end{equation}
	where $\dot{\lambda}$ denotes the time derivative  of $\lambda$.
 We then rewrite \eqref{eq:symmetry weak solution phi-lambda}
 as 
\begin{align}\label{eq:symmetry weak solution phi-lambda 2}
 0&=\langle u , -\partial_t(L(\phi_{\lambda})) + 2 \dot{\lambda} \partial_x (L(\phi_{\lambda})) \rangle + \langle  \kappa u +\frac{3}{2}  u^{2}+\frac{1}{2}  (\partial_x  u)^{2} , \partial_x^2\phi_{\lambda} \rangle  \nonumber \\
		&-  \langle \frac{1}{2} u^{2} , \partial_x^4\phi_{\lambda} \rangle+ \langle u, \partial_y^2\phi_{\lambda}  \rangle  .
 \end{align}
	Then, since $\phi \rightarrow \phi_{\lambda}$ is a bijection in $ C^{1}_{0}(I,C^{\infty}_{0}(\mathbb{R}^{2}))$ and  $(\phi_{\lambda})_{\lambda}=\phi$, we conclude that \eqref{eq:symmetry weak solution phi-lambda 2}
 actually
 holds with   $\phi_{\lambda}$ replaced by $\phi$, namely	
 \begin{equation}\label{eq:equation with symmetry}
		\langle u , - L (\phi_{t})+2 \dot{\lambda} \partial_x(L \phi) \rangle + \langle  \kappa u +\frac{3}{2}  u^{2} +\frac{1}{2}  (\partial_x u)^{2} , \partial_x^2 \phi \rangle - \langle \frac{1}{2} u^{2} , \partial_x^4 \phi \rangle+ \langle u, \partial_y^2 \phi  \rangle=0.
	\end{equation}
	Comparing \eqref{def:chkp weak solution} with \eqref{eq:equation with symmetry}, we get
	\begin{equation}\label{eq:inner chkp-finally}
		\langle u, \dot{\lambda} \partial_x(L \phi) \rangle + \langle  \kappa u +\frac{3}{2}  u^{2} +\frac{1}{2}  (\partial_x u)^{2}, \partial_x^2\phi \rangle -  \langle \frac{1}{2} u^{2} ,\partial_x^4 \phi \rangle+ \langle u, \partial_y^2\phi \rangle  =0 .
	\end{equation}
	We now fix a  $t_{0} \in I$. For any $\psi \in C^{\infty}_{0}(\mathbb{R}^{2})$, we consider the sequence of functions $\phi_{\varepsilon}(t,x,y) = \psi(x,y) \rho_{\varepsilon}(t)$, where $\rho_{\varepsilon} \in C^{\infty}_{0}(I)$ is a mollifier such that $\rho_{\varepsilon}(t) \rightarrow \delta(t-t_{0})$, the  Dirac distribution at $t_{0}$ as $\varepsilon \rightarrow 0$.
	Replacing $\phi$ by $\phi_{\varepsilon}(t,x,y)$ in \eqref{eq:inner chkp-finally} we obtain 
	\begin{align}\label{eq:chkp with sequence}
		&\iint_{ \mathbb{R} \times \mathbb{R}} \partial_x(L \psi) \int_{I}  \dot{\lambda} u \rho_{\varepsilon}(t)\, \mathrm{d}t\mathrm{d}x\mathrm{d}y +	\iint_{ \mathbb{R} \times \mathbb{R}} \partial_x^2\psi  \int_{I}  (\kappa u +\frac{3}{2}  u^{2} +\frac{1}{2}  (\partial_x u)^{2})\rho_{\varepsilon}(t)\, \mathrm{d}t\mathrm{d}x\mathrm{d}y \nonumber \\
		&-\iint_{ \mathbb{R} \times \mathbb{R}}  \partial_x^4\psi  \int_{I} \frac{1}{2} u^{2} \rho_{\varepsilon}(t) \, \mathrm{d}t\mathrm{d}x\mathrm{d}y+\iint_{ \mathbb{R} \times \mathbb{R}}  \partial_y^2\psi\int_{I}  u \rho_{\varepsilon}(t) \, \mathrm{d}t\mathrm{d}x\mathrm{d}y
		 =0
	\end{align}
	Since  $u \in C(I,H^{1}(\mathbb{R}^{2}))$ is a weak solution, it is clear that  
	\begin{equation*}
		\lim_{\varepsilon \to 0} \int_{I} u(t,x,y) \rho_{\varepsilon}(t) \, \mathrm{d}t = u(t_{0},x,y), \quad \mathrm{in} \quad L^{2}(\mathbb{R}^{2})
	\end{equation*}
	and
	\begin{align*}
		&\lim_{\varepsilon \to 0} \int_{I} u^{2}_{x}(t,x,y) \rho_{\varepsilon}(t) \, \mathrm{d}t = u^{2}_{x}(t_{0},x,y)  \quad \mathrm{in} \quad L^{1}(\mathbb{R}^{2}),
 \\
		&\lim_{\varepsilon \to 0} \int_{I} u^{2}(t,x,y) \rho_{\varepsilon}(t) \, \mathrm{d} t = u^{2}(t_{0},x,y)  \quad \mathrm{in} \quad L^{1}(\mathbb{R}^{2}).
	\end{align*}
As $\varepsilon$ tends to zero,  we deduce from \eqref{eq:chkp with sequence} that $u(t_{0},x,y)$ satisfies Lemma \ref{chkp-lemma} for $c=\dot{\lambda}(t_{0}) $ so that $\overline{u}(t,x,y):=u(t_{0},x-\dot{\lambda}(t_{0})(t-t_{0}),y)$ is a traveling wave solution in the $x$-direction. Since $\overline{u}(t_{0},x,y)=u(t_{0},x,y)$, the  uniqueness of solution implies that $\overline{u}(t,x,y)=u(t,x,y)$ for all time $t$. Therefore, $u$ is a traveling wave solution in the $x$-direction  with speed $\dot{\lambda}(t_{0})$.

\end{proof}

We now turn to the hyperelastic compressible plate model \eqref{eq:hcp} and consider the steadiness of weak  solutions with symmetry in the horizontal direction. As before, we first introduce the weak formulation of solutions.

\medskip

\begin{definition}\label{def:hcp weak solution}
	A function $u \in C(I,H^{1}(\mathbb{R}^{2}))$ is called a weak solution of \eqref{eq:hcp} if $u $ satisfies
	\begin{equation}
		\iiint_{I \times \mathbb{R}^{2}} u L(\partial_{t} \phi) +(\frac{3}{2} u^{2}+\frac{1}{2} \gamma (\partial_x u)^{2}) \partial^{2}_{x} \phi -  \frac{\gamma}{2} u^{2} \partial^{4}_{x} \phi- \alpha u \partial^{2}_{y} \phi+\beta u  \partial^{2}_{x} \partial^{2}_{y} \phi  \,  \mathrm{d}t\mathrm{d} x\mathrm{d} y=0,
	\end{equation}
	for all $\phi \in C^{\infty}_{0}(I \times \mathbb{R}^{2} )$.
\end{definition}

\medskip

As before, we also need to characterize steady solutions in the weak formulation, which is included in following lemma.

\medskip

\begin{lemma}\label{lem:hcp weak traveling solution}
	If $U\in  H^{1}(\mathbb{R}^{2}) $ satisfies
	\begin{equation}
		\iint_{\mathbb{R}^{2}} -c U L( \partial_{x} \psi) +( \frac{3}{2} U^{2}+\frac{1}{2} \gamma (\partial_x U)^{2}) \partial^{2}_{x} \psi - \frac{\gamma}{2} U^{2} \partial^{4}_{x} \psi- \alpha U \partial^{2}_{y} \psi+\beta U  \partial^{2}_{x} \partial^{2}_{y}  \psi \,  \mathrm{d}  x \mathrm{d} y=0,
	\end{equation}
	for all $\psi \in C^{\infty}_{0}(\mathbb{R}^{2}) $, then $u$ given by 
	\begin{equation}
		u(t,x,y)=U(x-c(t-t_0),y)
	\end{equation}
	is a weak solution of \eqref{eq:hcp} for any $t_0 \in I$.
\end{lemma}

\medskip

With Definition \ref{def:hcp weak solution} and Lemma \ref{lem:hcp weak traveling solution}, we can follow the above procedure for the CH-KP equation and similarly prove that weak solutions with symmetry  is actually steady in the horizontal direction. We will just state the result below and omit the proof.
\medskip

\begin{theorem}\label{thm: hcp steadiness of x symmetric weak solutions}
	Let $u$ be a weak solution of the hyperelastic compressible plate model \eqref{eq:hcp} with initial data $u_{0}(x,y)=u(t_{0},x,y)$ such that the model is locally well-posed \cite{chen3032}. If $u$ is x-symmetric, then it is a traveling wave in the x-direction.
\end{theorem}

%
%
%

\bmhead{Acknowledgments}
The authors are very grateful to the referee for the valuable comments and suggestions, which lead to the current version of the manuscript. The authors would like to thank Dr. Lu Li for the many discussions on the ideas and proofs during the preparation of the initial version, and for her careful proofreading before the submission.The author L.P. gratefully acknowledges financial support from the National Natural Science Foundation for Young Scientists of China (Grant
No. 12001553), the Fundamental Research Funds for the Central Universities (Grant No. 20lgpy151), the Science and Technology Program of Guangzhou (Grant No. 202102080474) and the Guangdong Basic, and Applied Basic Research Foundation (Grant No. 2023A1515010599).

\section*{Declarations}

 \textbf{Conflict of interest} 
	The authors declare that they have no conflict of interest.\\

\noindent\textbf{Associated data} Data sharing is not applicable to this article as no datasets were generated or analysed
during the current study.\\
	
\noindent Publisher's Note Springer Nature remains neutral with regard to jurisdictional claims in published maps
and institutional affiliations.\\

\noindent Springer Nature or its licensor (e.g. a society or other partner) holds exclusive rights to this article
under a publishing agreement with the author(s) or other rightsholder(s); author self-archiving of the
accepted manuscript version of this article is solely governed by the terms of such publishing agreement
and applicable law.

\end{document}